\documentclass[11pt]{amsart}
\usepackage[utf8]{inputenc}
\usepackage{mathrsfs} 
\usepackage{amssymb}
\usepackage{bm}
\usepackage{graphicx}
\usepackage[centertags]{amsmath}
\usepackage{amsfonts}
\usepackage{amsthm}
\usepackage{enumerate}
\usepackage{tocvsec2}
\usepackage{xcolor}
\usepackage[margin=1in]{geometry}

\newtheorem{theorem}{Theorem}[section]
\newtheorem*{theorem*}{Theorem}

\newtheorem{corollary}[theorem]{Corollary}
\newtheorem{lemma}[theorem]{Lemma}
\newtheorem{rem}[theorem]{Remark}

\newtheorem{proposition}[theorem]{Proposition}

\newtheorem{question}{Question}[section]

\theoremstyle{definition}
\newtheorem{definition}[theorem]{Definition}

\newcommand{\ee}{\varepsilon}
\newcommand{\nn}{\mathbb{N}}

\newcommand{\supp}{\text{supp}}

\begin{document}

\title{A $\xi$-weak Grothendieck  compactness principle}
\author{K. Beanland, R.M. Causey}

\begin{abstract} For $0\leqslant \xi\leqslant \omega_1$, we define the notion of $\xi$-weakly precompact and $\xi$-weakly compact sets in Banach spaces and prove that a set is $\xi$-weakly precompact if and only if its weak closure is $\xi$-weakly compact.  We prove a quantified version of Grothendieck's compactness principle and the characterization of Schur spaces obtained in \cite{DFLORT} and \cite{JLO}. For $0\leqslant \xi\leqslant \omega_1$,  we prove that a Banach space $X$ has the $\xi$-Schur property if and only if every $\xi$-weakly compact set is contained in the closed, convex hull of a weakly null (equivalently, norm null) sequence. The $\xi=0$ and $\xi=\omega_1$ cases of this theorem are the theorems of Grothendieck and \cite{DFLORT}, \cite{JLO}, respectively.

\end{abstract}

\maketitle

\section{Introduction}

The following theorem is known as the Grothendieck compactness principle.

\begin{theorem} Any norm compact subset of any Banach space lies in the closed, convex hull of a norm null sequence.

\label{Groth}
\end{theorem}

In \cite{DFLORT}, the authors considered the validity of the resulting statement when the norm topology is replaced by the weak topology. We recall that a Banach space has the Schur property if weakly compact sets coincide with norm compact sets (equivalently,  weakly null sequences coincide with norm null sequences). This, combined with Grothendieck's theorem, gives the easy direction of the following result.

\begin{theorem} Given a Banach space $X$, every weakly compact set in $X$ lies in the closed, convex hull of a weakly null sequence in $X$ if and only if $X$ has the Schur property. 

\label{non}
\end{theorem}

The hard direction of Theorem \ref{non} was shown in \cite{DFLORT} by basic sequence techniques. Later, in \cite{JLO}, a second proof was given using an operator theoretic approach.

For $0\leqslant \xi\leqslant \omega_1$, using the notion of $\xi$-weak nullity defined in \cite{AMerTs-Israel}, for each Banach space $X$, we will define a topology $\tau_\xi$ on $X$. We will  study the compact and precompact sets in this topology, which we call $\xi$-weakly compact and $\xi$-weakly precompact sets. We give all necessary definitions in Section $2$.    

The $0$-weakly null sequences are precisely the norm null sequences, from which it will follow that the $0$-weakly (pre)compact sets are precisely the are precisely the norm (pre)compact sets. The $\omega_1$-weakly null sequences are precisely the weakly null sequences. Combining this fact with the Eberlein-\v{S}mulian theorem, the $\omega_1$-weakly (pre)compact sets are precisely the weakly (pre)compact sets.

A Banach space is said to have the $\xi$-Schur property if every $\xi$-weakly null sequence in $X$ is norm null. Every Banach space is $0$-Schur, and $\omega_1$-Schur spaces are precisely  the Schur spaces. For $0\leqslant \xi\leqslant \omega_1$, let $\textsf{V}_\xi$ denote the class of $\xi$-Schur Banach spaces.    These classes are such that if $\xi\leqslant \zeta\leqslant \omega_1$, then $\textsf{V}_\zeta\subset \textsf{V}_\xi$, and the classes $$\textsf{V}_0, \textsf{V}_{\omega^\xi}, 0\leqslant \xi\leqslant \omega_1$$ are all distinct. We remark that this is a complete list of distinct classes from the list $\textsf{V}_\xi$, $0\leqslant \xi\leqslant \omega_1$. More precisely, if $\omega^\xi<\zeta<\omega^{\xi+1}$, $\textsf{V}_\zeta=\textsf{V}_{\omega^\xi}$. 

Our main result is to prove the $0<\xi<\omega_1$ cases of the following theorem, which interpolates between Theorem \ref{Groth}, the $\xi=0$ case, and Theorem \ref{non}, the $\xi=\omega_1$ case. 

\begin{theorem} Let $X$ be a Banach space and let $\xi\leqslant \omega_1$ be an ordinal. Then every $\xi$-weakly compact subset of $X$ lies in the closed, convex hull of a $\xi$-weakly null sequence in $X$ if and only if $X$ is $\xi$-Schur. 

\label{main}

\end{theorem}

\begin{rem}\upshape We note that we will actually prove something stronger than what is stated in the theorem. We will prove that for $0<\xi<\omega_1$ and a Banach space $X$, the following are equivalent. 

\begin{enumerate}[(i)]\item $X$ has the $\xi$-Schur property.  \item Every $\xi$-weakly compact subset of $X$ is contained in the closed, convex hull of a $0$-weakly null (norm null) sequence.  \item Every $\xi$-weakly compact subset of $X$ is contained in the closed, convex hull of a $\xi$-weakly null sequence.  \item Every $\xi$-weakly compact subset of $X$  is contained in the closed, convex hull of an $\omega_1$-weakly null (weakly null)  sequence.  \end{enumerate}
\label{mer}
\end{rem}

With the aid of Theorem \ref{Groth} and the nature of the weakly null hierarchy, it is easy to see that $(i)\Rightarrow (ii)\Rightarrow (iii)\Rightarrow (iv)$. In Section 2, we provide the requisite definitions and background results and prove $(iv)\Rightarrow (i)$, completing the proof of the equivalences.

In Section 3, we define the $\tau_\xi$ topology and prove the following topological results.  

\begin{theorem} For each Banach space $X$ and $0\leqslant \xi<\omega_1$, there exists a topology $\tau_\xi$ on $X$ such that the $\xi$-weakly compact subsets of $X$ are the compact subsets of $X$ in the $\tau_\xi$ topology, and the $\xi$-weakly precompact subsets of $X$ are the relatively compact subsets of $X$ in the $\tau_\xi$ topology. 

\label{topology}

\end{theorem}

In Section 4, we use a construction from \cite{LA} to investigate  convex hulls and $\xi$-weak compactness. We show that for $0<\xi<\omega$, in the $\tau_\xi$ topology,  there exist Banach spaces in which the $\tau_\xi$ topology is not locally convex, and in which the closed, convex hull of a compact set need not be compact.

Section 5 concludes with several examples to show the richness of the classes of $\xi$-Schur spaces.

\section{Definitions}

Throughout, all Banach spaces will be over the scalar field $\mathbb{K}$, which is either the real or complex numbers. 

Throughout, if $M$ is an infinite subset of $\nn$, $[M]^{<\omega}$ (resp. $[M]$) denotes the set of all finite (resp. infinite) subsets of $M$.  We identify each subset $E$ of $\nn$ with the sequence obtained by listing the members of $E$ in strictly increasing order. We let $E<F$ denote the relation that $\max E<\min F$, with the convention that $\max \varnothing=0$ and $\min \varnothing=\infty$. We let $E\preceq F$ denote the relation that $E$ (treated as a sequence) is an initial segment of $F$.  We let $E\prec F$ denote the concatenation of $E$ with $F$.  We say a subset $\mathcal{F}$ of $[\nn]^{<\omega}$ is \begin{enumerate}[(i)]\item \emph{hereditary} if $E\subset F\in \mathcal{F}$ implies $E\in \mathcal{F}$, \item \emph{spreading} if $(m_i)_{i=1}^t\in \mathcal{F}$ and $m_i\leqslant n_i$ implies $(n_i)_{i=1}^t\in \mathcal{F}$,  \item \emph{compact} if $\{1_E: E\in \mathcal{F}\}$ is compact in $\{0,1\}^\nn$ with its product topology, \item \emph{regular} if it is hereditary, spreading, and compact.  \end{enumerate} We say a sequence $(E_i)_{i=1}^t$ of finite subsets of $\nn$ is $\mathcal{F}$-\emph{admissible} provided that $E_1<\ldots <E_t$ and $(\min E_i)_{i=1}^t\in \mathcal{F}$.

We recall the \emph{Schreier families} $\mathcal{S}_\xi$, $\xi<\omega_1$.  We let $$\mathcal{S}_0=\{\varnothing\}\cup \{(n): n\in\nn\},$$ $$\mathcal{S}_{\xi+1}= \{\varnothing\}\cup \Bigl\{\bigcup_{i=1}^n E_i: E_1<\ldots <E_n, n\leqslant \min E_1, E_i\in \mathcal{S}_\xi\Bigr\},$$ and if $\xi$ is a limit ordinal, there exists $\xi_n\uparrow \xi$ such that $$\mathcal{S}_\xi=\{E: (\exists n\in \nn)(n\leqslant E\in \mathcal{S}_{\xi_n}\}.$$  We also recall the Schreier spaces $X_\xi$.   The space $X_\xi$ is the completion of $c_{00}$ with respect to the norm $$\|x\|=\sup \Bigl\{\sum_{n\in E} |x_n|: E\in\mathcal{S}_\xi\}.$$  Here, $c_{00}$ is the space of all finitely supported scalar sequences and $x=(x_n)_{n=1}^\infty$. 

For $0<\xi<\omega_1$, we say a sequence $(x_n)_{n=1}^\infty$ in a Banach space is an $\ell_1^\xi$-\emph{spreading model} if $$0< \inf \{\|x\|: E\in \mathcal{S}_\xi, x=\sum_{n\in E}a_n x_n, 1=\sum_{n\in E}|a_n|\}.$$  It is easy to see that for $0<\xi<\omega_1$,  the canonical $c_{00}$ basis is an $\ell_1^\xi$-spreading model in $X_\xi$.  We recall that a sequence $(x_n)_{n=1}^\infty$ in a Banach space $X$ is called \emph{convexly unconditional} if for every $\delta>0$, there exists $C(\delta)>0$ such that for any $(w_n)_{n=1}^\infty \in c_{00}$ with $\sum_{n=1}^\infty |w_n|\leqslant 1$ and $\|\sum_{n=1}^\infty w_nx_n\|\geqslant \delta$, $\|\sum_{n=1}^\infty \lambda_n w_nx_n\|\geqslant C(\delta)$ for any $(\lambda_n)_{n=1}^\infty$ with $|\lambda_n|=1$ for all $n\in\nn$.  It was shown in \cite{AMerTs-Israel} that any seminormalized, weakly null sequence has a convexly unconditional subsequence. We note that for $0<\xi<\omega_1$, if $(x_n)_{n=1}^\infty$ is convexly unconditional and $$0<\inf\{\|x\|: E\in \mathcal{S}_\xi, x\in \text{co}(x_n: n\in E)\},$$ then $(x_n)_{n=1}^\infty$ is an $\ell_1^\xi$-spreading model.

For $0<\xi<\omega_1$, we say a sequence $(x_n)_{n=1}^\infty\in X$ is $\xi$-\emph{weakly convergent to} $x\in X$ if it is weakly convergent to $x$ and  no subsequence of $(x_n-x)_{n=1}^\infty$ is an $\ell_1^\xi$-spreading model.  We say $(x_n)_{n=1}^\infty$ is $0$-\emph{weakly convergent} (resp. $\omega_1$-\emph{weakly convergent}) to $x\in X$ if it is norm (resp. weakly) convergent to $x$.  Note that if $(x_n)_{n=1}^\infty$ is $\xi$-weakly convergent to $x$, it is also $\zeta$-weakly convergent to $x$ for all $\xi\leqslant \zeta\leqslant \omega_1$. We note that in \cite{AGodR-book}, $\xi$-\emph{convergent} was defined (see Definition III.2.7) using the repeated averages hierarchy. It follows from \cite[Proposition III.3.10]{AGodR-book} that the notion of $\xi$-convergent defined there coincides with what we have defined here as $\xi$-weakly convergent to $0$.

For $0\leqslant \xi\leqslant \omega_1$, we say a sequence is $\xi$-\emph{weakly convergent} if it is $\xi$-weakly convergent to some $x$. We say  sequence is $\xi$-\emph{weakly null} if it is $\xi$-weakly convergent to $0$.   

\begin{definition}
Given a subset $K$ of a Banach space $X$, we say $K$ is $\xi$-\emph{weakly precompact} (resp. $\xi$-\emph{weakly compact}) if every sequence  $(x_n)_{n=1}^\infty\subset K$ has a subsequence which is $\xi$-weakly convergent to some member of $X$ (resp .$K$).  Note that $0$-weak (pre)compactness  coincides with norm (pre)compactness,  and $\omega_1$-weak (pre)compactness coincides with weak (pre)compactness. 
\label{def}
\end{definition}

We isolate the following criterion to check that a given set is not $\xi$-weakly precompact. 

\begin{proposition} Fix $0<\xi<\omega_1$ and a subset $K$ of the Banach space $X$.  If there exist a sequence $(y_n)_{n=1}^\infty$ of $K$ and $y\in X$ such that $(y_n)_{n=1}^\infty$ is weakly convergent to $y$ and $$0<\inf \{\|y'-y\|: E\in \mathcal{S}_\xi, y'\in \text{\emph{co}}(y_n: n\in E)\},$$ then $K$ is not $\xi$-weakly precompact. 
\label{crit}
\end{proposition}

\begin{proof} Since $(y_n-y)_{n=1}^\infty$ is seminormalized and weakly null, then by passing to a subsequence and relabeling, we may assume $(y_n-y)_{n=1}^\infty$ is convexly unconditional.  Then $$0<\inf \{\|y'-y\|: E\in \mathcal{S}_\xi, y'\in \text{co}(y_n: n\in E)\}$$ implies that $$0<\inf \{\|y'-y\|: E\in \mathcal{S}_\xi, y'=\sum_{i\in E} a_iy_i, \sum_{i\in E}|a_i|=1\} .$$   This implies that $(y_n-y)_{n=1}^\infty$ is an $\ell_1^\xi$-spreading model, as are all of its subsequences. This implies that no subsequence of $(y_n)_{n=1}^\infty$ is $\xi$-weakly convergent to any $x\in X$. This is because the only $x\in X$ to which a subsequence of $(y_n)_{n=1}^\infty$ could be $\xi$-weakly convergent is $x=y$. However, for any subsequence $(y'_n)_{n=1}^\infty$ of $(y_n)_{n=1}^\infty$, $(y'_n-y)_{n=1}^\infty$ is an $\ell_1^\xi$-spreading model, and so $(y'_n)_{n=1}^\infty$ is not $\xi$-weakly convergent to $y$.

\end{proof}

\begin{proposition} For a convex subset $K$ of the Banach space $X$, $K$ is $\xi$-weakly precompact if and only if $\overline{K}^\text{\emph{weak}}$ is $\xi$-weakly compact. 

\label{barbarian}
\end{proposition}

\begin{proof} If $\overline{K}^\text{weak}$ is $\xi$-weakly compact, then any sequence $(x_n)_{n=1}^\infty\subset K\subset \overline{K}^\text{weak}$ has a subsequence which is $\xi$-weakly convergent, and therefore weakly convergent,  to some $x\in \overline{K}^\text{weak}\subset X$. Therefore $K$ is $\xi$-weakly precompact. 

For the converse, note that since $K$ is convex, $\overline{K}^\text{weak}$ is just the norm closure $\overline{K}$ of $K$ by the Mazur lemma. Assume $\overline{K}$ is not $\xi$-weakly compact. We will prove that $K$ is not $\xi$-weakly precompact. 

If $\overline{K}$ is not weakly compact, then $K$ is not weakly precompact, and therefore not $\xi$-weakly precompact. 

Assume $\overline{K}$ is weakly compact. Since it is not $\xi$-weakly compact, there exists sequence $(x_n)_{n=1}^\infty\subset \overline{K}$ with no $\xi$-weakly convergent subsequence. Since $\overline{K}$ is weakly compact, we may pass to a subsequence and assume there exists $x\in \overline{K}$ such that $(x_n-x)_{n=1}^\infty$ is weakly null. Since $(x_n)_{n=1}^\infty$ has no $\xi$-weakly convergent subsequence, $(x_n-x)_{n=1}^\infty$ has no $\xi$-weakly null subsequence, and in particular it is weakly null and not $\xi$-weakly null. By the definition of $\xi$-weak nullity, $(x_n-x)_{n=1}^\infty$ has a subsequence which is an $\ell_1^\xi$-spreading model. We may pass to a subsequence and relabel and assume $(x_n-x)_{n=1}^\infty$ is an $\ell_1^\xi$-spreading model. This means there exists $\ee>0$ such that $$\ee \leqslant \inf \{\|y-x\|: E\in \mathcal{S}_\xi, y=\sum_{i\in E}a_ix_i, \sum_{i\in E}|a_i|=1\}.$$ We can select $(y_n)_{n=1}^\infty\subset K$ such that for all $n\in\nn$, $\|x_n-y_n\|<\min \{n^{-1}, \ee/2\}$. Then $(y_n-x)_{n=1}^\infty$ is also a weakly null, $\ell_1^\xi$-spreading model. Weak nullity follows from the fact that $\lim_n \|(y_n-x)-(x_n-x)\|=0$, while the fact that $(y_n-x)_{n=1}^\infty$ is an $\ell_1^\xi$-spreading model follows from the fact that for any $E\in\mathcal{S}_\xi$ and scalars $(a_i)_{i\in E}$ with $\sum_{i\in E}|a_i|=1$, $$\|\sum_{i\in E}a_i(y_i-x)\|\geqslant \|\sum_{i\in E}a_i(x_i-x)\|-\sum_{i\in E}|a_i|\|y_i-x_i\| \geqslant \ee-\ee/2=\ee/2.$$   An appeal to Proposition \ref{crit} yields that $K$ is not $\xi$-weakly precompact.

\end{proof}

Next we prove a result regarding convexity. In Section $4$, we discuss the sharpness of this result.

\begin{proposition} For $0<\xi<\omega_1$ and a Banach space $X$, if $(x_n)_{n=1}^\infty\subset X$ is $\xi$-weakly null, then $(x_n)_{n=1}^\infty$ admits a subsequence $(x_{k_n})_{n=1}^\infty$ such that the closed, convex hull $\overline{\text{\emph{co}}}(x_{k_n}:n\in\nn)$ is $\xi$-weakly compact.

\label{convex}
\end{proposition}

\begin{proof} We recall from \cite{BCFrWa-JFA} the definition of the operator ideal $\mathfrak{W}_\xi$ such that $\mathfrak{W}_\xi(E,F)$ is the set of all operators $A:E\to F$ such that for any $(e_n)_{n=1}^\infty\subset B_E$, $(Ae_n)_{n=1}^\infty$ has a subsequence which is $\xi$-convergent in $F$. This is equivalent to saying that $AB_E$ is $\xi$-weakly precompact.   

 It is implicitly contained within the proofs of Corollaries $3.11$ and $3.13$ of \cite{CN} that if $(x_n)_{n=1}^\infty$ is $\xi$-weakly null, there exist $k_1<k_2<\ldots$ such that the operator $\Phi:\ell_1\to X$ given by $\Phi\sum_{n=1}^\infty a_ne_n=\sum_{n=1}^\infty a_n x_{k_n}$ lies in $\mathfrak{W}_\xi$.  By Proposition \ref{barbarian}, $\overline{\Phi B_{\ell_1}}$ is $\xi$-weakly compact, and so is the weakly closed subset $\overline{\text{co}}(x_{k_n}: n\in\nn)$ of $\overline{\Phi B_{\ell_1}}$.

\end{proof}

\begin{rem}\upshape In the preceding proposition, passing to a subsequence is necessary, at least in the cases $0<\xi<\omega$. That is, for such $\xi$, there exists a Banach space and a $\xi$-weakly null sequence whose closed, convex hull is not $\xi$-weakly compact. We discuss this in the next section. It is this obstacle which prevents one from applying the argument from \cite{DFLORT} for the $\xi=\omega_1$ case verbatim to the $0<\xi<\omega_1$ cases.

\end{rem}

\begin{proof}[Proof of Theorem \ref{main}]  For $0<\xi<\omega_1$, we prove by contraposition that item (iv) of Remark \ref{mer} implies item (i) of Remark \ref{mer}.  Assume that $X$ is not $\xi$-Schur and let $(x_n)_{n=1}^\infty$ be a normalized, $\xi$-weakly null sequence. Using Proposition \ref{convex}, after passing to a subsequence and relabeling, we may assume that $K_0=\overline{\text{co}}(x_n:n\in\nn)$ is $\xi$-weakly compact.  For each $n\in\nn$, let $K_n= nK_0 \cap \frac{1}{n}B_X$.  Then each $K_n$ is $\xi$-weakly compact, since it is a weakly closed subset of the $\xi$-weakly compact set $nK_0$. Since $$\{0\}=\bigcap_{n=1}^\infty \bigcup_{m=n}^\infty K_m,$$ it is easy to see that $$K=\bigcup_{n=1}^\infty K_n$$ is also $\xi$-weakly compact. From here, the proof of \cite[Theorem 1]{DFLORT} may be followed exactly. 

\end{proof}

\begin{rem}\upshape Let us recall that if $\textsf{\emph{V}}_\xi$ denotes the class of $\xi$-Schur Banach spaces, $\textsf{\emph{V}}_{\omega^\xi}$, $\xi\leqslant \omega_1$, are all distinct. Therefore the preceding theorem interpolates between Theorem \ref{Groth} and Theorem \ref{non} in a non-vacuous way.

\end{rem}

\section{Topology}

In this section, for a subset $H$ of $\mathcal{S}_\xi$, we let $MAX(H)$ denote the set of members of $H$ which are maximal with respect to inclusion.  

\begin{proposition} For $\xi<\omega_1$, there exists a partition $(H_\zeta)_{\zeta\leqslant \omega^\xi}$ of $\mathcal{S}_\xi$ such that $H_0=MAX(\mathcal{S}_\xi)$ and if $H_\zeta\ni E\prec F$, $F\in \cup_{\mu<\zeta}H_\mu$.

\label{easy}
\end{proposition}

\begin{proof} Define $$\mathcal{S}_\xi^0=\mathcal{S}_\xi,$$ $$\mathcal{S}_\xi^{\zeta+1}= \mathcal{S}_\xi^\zeta\setminus MAX(\mathcal{S}_\xi^\zeta),$$ and if $\zeta<\omega_1$ is a limit ordinal, $$\mathcal{S}_\xi^\zeta = \bigcap_{\mu<\zeta} \mathcal{S}_\xi^\mu.$$  Note that $(\mathcal{S}_\xi^\zeta)_{\zeta\leqslant \omega^\xi}$ are decreasing with $\zeta$.  For each $\zeta\leqslant \omega^\xi$, let $$H_\zeta= MAX(\mathcal{S}_\xi^\zeta)= \mathcal{S}_\xi^\zeta \setminus \mathcal{S}_\xi^{\zeta+1}.$$ It follows from \cite[Proposition 3.2]{Concerning} that $(H_\zeta)_{\zeta\leqslant \omega^\xi}$ is a partition of $\mathcal{S}_\xi$.  It follows from the definition that $H_0=MAX(\mathcal{S}_\xi)$.  If $E\in H_\zeta$, then $E$ is maximal in $\mathcal{S}_\xi^\zeta$ with respect to inclusion. Therefore if $E\prec F$, $F\notin \mathcal{S}_\xi^\zeta$. But since $$\bigcup_{\zeta\leqslant \mu\leqslant \omega^\xi} H_\mu \subset \bigcup_{\zeta\leqslant \mu\leqslant \omega^\xi}\mathcal{S}_\xi^\mu \subset \mathcal{S}_\xi^\zeta,$$ this implies that $F\in H_\mu$ for some $\mu<\zeta$.

\end{proof}

We next provide a strengthening of Proposition \ref{barbarian}.

\begin{proposition} $K$ is $\xi$-weakly precompact if and only if $\overline{K}^\text{\emph{weak}}$ is $\xi$-weakly compact. 

\label{conan}
\end{proposition}

\begin{proof} As in the proof of Proposition \ref{barbarian}, if $\overline{K}^\text{weak}$ is $\xi$-weakly compact, then $K$ is $\xi$-weakly precompact, and if $\overline{K}^\text{weak}$ is not weakly compact, then $K$ is not $\xi$-weakly precompact. Convexity was not used in either of those arguments. What remains to show is that if $\overline{K}^\text{weak}$ is weakly compact but not $\xi$-weakly compact, then $K$ is not $\xi$-weakly precompact. As in the proof of Proposition \ref{barbarian}, we may assume there exist $(x_n)_{n=1}^\infty\subset \overline{K}^\text{weak}$, $x\in \overline{K}^\text{weak}$, and $\ee>0$ such that $(x_n-x)_{n=1}^\infty$ is weakly null and  $$\ee< \inf \{\|y-x\|: E\in \mathcal{S}_\xi, y=\sum_{i\in E}a_ix_i, \sum_{i\in E}|a_i|=1.\}.$$   

Let $u_n=x_n-x$. Fix a free ultrafilter $\mathcal{U}$ on $\nn$.  By the geometric Hahn-Banach theorem, for every $E\in MAX(\mathcal{S}_\xi)$, we may find $f_E\in B_{X^*}$ such that for every $n\in E$, $\text{Re\ }f_E(u_n)\geqslant \ee$.  Let $(H_\zeta)_{\zeta\leqslant \omega^\xi}$ be the partition from Proposition \ref{easy}.  We now define $g_E\in B_{X^*}$ for all $E\in \mathcal{S}_\xi$. We define $g_E$ for $E\in H_\zeta$ by induction on $\zeta$.  For $E\in H_0=MAX(\mathcal{S}_\xi)$, let $g_E=f_E$.   Now suppose $E\in H_\zeta$ for some $\zeta>0$ and if $g_F$ has been defined for all $F\in \cup_{\mu<\zeta}H_\mu$.  Since $E$ is non-maximal in $\mathcal{S}_\xi$, by standard properties of $\mathcal{S}_\xi$, $E\smallfrown(n)\in \mathcal{S}_\xi$ for all $n>\max E$. By the properties of $(H_\mu)_{\mu\leqslant \omega^\xi}$, for all $n>\max E$, $E\smallfrown (n)\in \cup_{\mu<\zeta} H_\mu$.  We now define $$g_E=\underset{n\in\mathcal{U}} \lim g_{E\smallfrown(n)},$$ where the limit is taken in the weak$^*$-topology.    This completes the definition of $g_E$ for all $E\in \mathcal{S}_\xi$.  We observe the following facts, easily shown for $E\in H_\zeta$ by induction on $\zeta$.    For any $E\in \mathcal{S}_\xi$, $$g_E\in \overline{\{f_F: F\in MAX(\mathcal{S}_\xi): E\preceq F\}}^{\text{weak}^*}\subset B_{X^*}.$$  In particular, if $n\in E$, then $n\in F$ for any $E\preceq F\in MAX(\mathcal{S}_\xi)$, so $\text{Re\ }g_E(u_n)\geqslant \ee$.

Fix a decreasing sequence of positive numbers $(\ee_n)_{n=0}^\infty\subset (0,1)$ such that $\ee-3\sum_{j=0}^\infty j\ee_j >\ee_0$. Let us recursively define $T_n\subset \mathcal{S}_\xi$, $M_1\supset M_2\supset \ldots$, $M_n\in \mathcal{U}$,  $m_1<m_2<\ldots$, $m_n\in M_n$,  and $v_n\in K-x$ recursively.   We perform the $n=1$ and $n=p+1>1$ cases simultaneously. In the $n=1$ case, let $p=0$.  If $n=p+1>1$, assume all previous choices have already been made.   Define $$T_n=\{E\in \mathcal{S}_\xi: E\subset \{1, \ldots, m_p\}\},$$ where $\{1, \ldots, m_0\}=\varnothing$ in the $p=0$, $n=1$ case.  Since for every $E\in T_n$ and every $1\leqslant j\leqslant p$, $$0=\underset{l\in \mathcal{U}}{\lim} (g_{E\smallfrown(l)}-g_E)(v_{m_j})=\underset{l\in\mathcal{U}}{\lim} g_E(u_l), $$ we may select $M_{n+1}\in\mathcal{U}$ such that $$|(g_{E\smallfrown (l)}-g_E)(v_{m_j})|<\ee_{|E|+1}$$ and $$|g_E(u_l)|<\ee_n$$ for all $E\in T_n$ and $1\leqslant j\leqslant p$.    By replacing $M_{n+1}$ with $M_{n+1}\cap M_n\in \mathcal{U}$, we may assume $M_{n+1}\subset M_n$.  We choose $m_n\in M_n$ arbitrary if $n=1$, and if $n>1$,  choose $m_n\in M_n$ such that $m_n>m_p$.     Since $u_{m_n}\in \overline{K}^\text{weak}-x$, $\text{Re\ }(u_{m_n})\geqslant \ee$ for all $E\in \mathcal{S}_\xi$ with $\max E=m_n$, and $|g_E(u_{m_n})|<\ee_n$ for all $E\in T_n$, we may choose $v_n\in K-x$ such that $\text{Re\ }g_E(v_n)\geqslant \ee-\ee_0$ for all $E\in \mathcal{S}_\xi$ with $\max E=m_n$ and such that  $|g_E(v_n)|<\ee_n$ and $|g_E(v_n-u_{m_n})|<\ee_n$ for all $E\in T_n$.  This completes the construction. Let $M=(m_n)_{n=1}^\infty$. Let us collect the important features of this construction. \begin{enumerate}[(i)]\item If $m_l\in E\in \mathcal{S}_\xi \cap [M]^{<\omega}$ and $m_l<\max E$, $|(g_E-g_{E\setminus \{\max E\}})(v_l)|<\ee_{|E|}$.  \item If $E\in \mathcal{S}_\xi \cap [M]^{<\omega}$, $m_l=\max E$, and $l<n$, then $|g_E(v_n)|<\ee_n$,  \item If $E\in \mathcal{S}_\xi\cap [M]^{<\omega}$ and $m_n=\max E$, $|g_E(v_n)|\geqslant \ee-\ee_0$.  \end{enumerate}

Now fix $\varnothing\neq E=\{m_{p_1}, \ldots, m_{p_t}\}\in \mathcal{S}_\xi\cap [M]^{<\omega}$ with $p_1<\ldots <p_t$.   Let  $E_0=\varnothing$ and $E_i=\{m_{p_1}, \ldots, m_{p_i}\}$.    We may write $f_E=g_\varnothing+\sum_{i=1}^t (g_{E_i}-g_{E_{i-1}})$. Note also that for $1\leqslant i\leqslant t$, $E_i\setminus \{\max E_i\}=E_{i-1}$.     For $1\leqslant j\leqslant t$, \begin{align*} \text{Re\ }f_E(v_{p_j}) & \geqslant  \text{Re\ }(g_{E_j}-g_{E_{j-1}})(v_{p_j}) - |g_\varnothing(v_{p_j})| - \sum_{i=1}^{j-1} |(g_{E_i}-g_{E_{i-1}})(v_{p_j})| - \sum_{i=j+1}^t |(g_{E_i}-g_{E_{i-1}})(v_{p_j})| \\ &  \geqslant \text{Re\ }g_{E_j}(v_{p_j}) - 2\sum_{i=0}^{j-1} |g_{E_i}(v_{p_j})| - \sum_{i=j+1}^t |(g_{E_i}-g_{E_{i-1}})(v_{p_j})| \\ & \geqslant \ee - \ee_0 -2j\ee_{p_j} - \sum_{i=j+1}^\infty \ee_j >\ee_0.\end{align*}   

Now since $v_n\in K-x$, $y_n=v_n+x\in K$. Since $K$ is weakly compact, we may select $p_1<p_2<\ldots$ and $y\in X$ such that $(y_{p_n})_{n=1}^\infty$ is weakly convergent to $y$ and let $P=\{m_{p_1}, m_{p_2}, \ldots\}$.  We now claim that $g_E(y)=g_E(x)$ for all $E\in \mathcal{S}_\xi\cap [P]^{<\omega}$.   To see this, fix such an $E=\{m_{p_1}, \ldots, m_{p_l}\}$ and note that, since $(y_{p_n})_{n=1}^\infty$ is weakly convergent to $y$ and $(x_{m_{p_n}})_{n=1}^\infty$ is weakly convergent to $x$,  \begin{align*} | g_E(y-x)| & = \underset{n}{\lim\sup}|g_E(y_{p_n}-x_{m_{p_n}})| = \underset{n}{\lim\sup} |g_E(v_{p_n}-u_{m_{p_n}})| \leqslant \underset{n}{\lim\sup} \ee_{p_n}=0. \end{align*}  

Now fix $E\in \mathcal{S}_\xi$ and $y'=\sum_{i\in E} a_i y_{p_i}\in \text{co}(y_{p_i}: i\in E)$. Note that $y'-x=\sum_{i\in E}a_iv_{p_i}$ and $F=(p_i)_{i\in E}\in \mathcal{S}_\xi$, from which it follows that \begin{align*} \|y'-y\| & \geqslant \text{Re\ }g_F(y'-y) = \text{Re\ }g_F(y'-x) = \sum_{i\in E}a_i \text{Re\ }g_F(v_{p_i}) \geqslant \ee_0\sum_{i\in E}a_i=\ee_0. \end{align*} Since $(v_{p_i})_{i=1}^\infty\subset K$, an appeal to Proposition \ref{crit} yields that $K$ is not $\xi$-weakly precompact. 

\end{proof}

We now introduce for $0<\xi<\omega_1$ the $\tau_\xi$ topology of a  Banach space $X$ and deduce that the $\xi$-weakly compact (resp. $\xi$-weakly precompact) subsets of $X$ are precisely the compact (resp. relatively compact) sets in the $\tau_\xi$ topology.  Most of the proofs in the remainder of this section follow from standard techniques, so we omit them. 

Let us say a subset $C$ of the Banach space is $\xi$-\emph{closed} if whenever $(x_n)_{n=1}^\infty\subset C$ is $\xi$-weakly convergent to $x\in X$, then $x\in C$. We say $U\subset X$ is $\xi$-\emph{open} if its complement is $\xi$-closed. We let $\tau_\xi$ denote the set of $\xi$-open subsets of $X$. 

\begin{lemma} \begin{enumerate}[(i)]\item $\tau_\xi$ is a topology on $X$. \item The $\tau_\xi$ topology is finer than the weak topology and coarser than the norm topology. In particular, it is a Hausdorff topology on $X$. \item If $(x_n)_{n=1}^\infty\subset X$ has no $\xi$-weakly convergent subsequence, then $T=\{x_n:n\in\nn\}$ is $\xi$-closed. \item A sequence $(x_n)_{n=1}^\infty$ is $\xi$-weakly convergent to $x\in X$ if and only if it is convergent to $x$ in the $\tau_\xi$ topology.  \end{enumerate}

\label{topology2}
\end{lemma}

\begin{proof} $(i)$ It is clear that $\varnothing,X$ are $\xi$-closed, as is an arbitrary intersection of $\xi$-closed sets.  Suppose $K_1, K_2$ are $\xi$-closed. We will show that $K_1\cup K_2$ is $\xi$-closed. To that end, suppose that $(x_n)_{n=1}^\infty\subset K_1\cup K_2$ is $\xi$-convergent to $x$. By passing to a subsequence, which is necessarily also $\xi$-convergent to $x$, we may assume there exists $j\in \{1,2\}$ such that $(x_n)_{n=1}^\infty\subset K_j$. By $\xi$-closedness of $K_j$, $x\in K_j\subset K_1\cup K_2$. 

$(ii)$ We must show that if $K\subset X$ is weakly closed, it is $\xi$-closed, and if it is $\xi$-closed, it is norm closed. Assume $K$ is $\xi$-closed and $(x_n)_{n=1}^\infty\subset K$ is $\xi$-convergent to $x\in X$. Then $x\in \overline{K}^\text{weak}=K$, and $K$ is $\xi$-closed.  Now suppose that $K$ is $\xi$-closed and $(x_n)_{n=1}^\infty\subset K$ is norm convergent to $x\in K$. Then $(x_n)_{n=1}^\infty$ is $\xi$-convergent to $x$. Since $K$ is $\xi$-closed, $x\in K$, and $K$ is norm closed. That $\tau_\xi$ is Hausdorff follows from the fact that the weak topology on $X$ is Hausdorff and $\tau_\xi$ is finer than the weak topology. 

$(iii)$ If $(y_n)_{n=1}^\infty\subset T$ is any sequence in $T$, it must either have a constant subsequence or a subsequence which is also a subsequence of $(x_n)_{n=1}^\infty$. Therefore if $(y_n)_{n=1}^\infty\subset T$ is $\xi$-convergent to some $x\in X$, then it must have a constant subsequence, and therefore $x$ is a member of the sequence $(x_n)_{n=1}^\infty$ and a member of $T$. Here we have used that limits of $\tau_\xi$-convergent sequences are unique because $\tau_\xi$ is a Hausdorff topology.

$(iv)$ Seeking a contradiction, assume that $(x_n)_{n=1}^\infty$ is $\xi$-convergent to $x$ but not convergent to $x$ in the $\tau_\xi$ topology. Then there exist a $\tau_\xi$-open set $U$ containing $x$ and subsequence $(y_n)_{n=1}^\infty$ of $(x_n)_{n=1}^\infty$ such that $(y_n)_{n=1}^\infty \subset X\setminus U$.  But since any subsequence of a $\xi$-convergent sequence is $\xi$-convergent to the same limit, $(y_n)_{n=1}^\infty\subset X\setminus U$ is $\xi$-convergent to $x\in U$. But this is impossible, since $X\setminus U$ is $\tau_\xi$-closed. 

Now suppose that $(x_n)_{n=1}^\infty$ is not $\xi$-convergent to $x$. Then either $(x_n)_{n=1}^\infty$ is not weakly convergent to $x$, or $(x_n-x)_{n=1}^\infty$ is weakly null but not $\xi$-weakly null. In the first case, there exist a weakly open $U$ containing $x$ and a subsequence $(y_n)_{n=1}^\infty$ of $(x_n)_{n=1}^\infty$ such that $(y_n)_{n=1}^\infty \subset X\setminus U$. Since $U$ is weakly open, it is also $\tau_\xi$-open, and $X\setminus U$ is $\tau_\xi$ closed.   Since $(y_n)_{n=1}^\infty\subset X\setminus U$, neither $(y_n)_{n=1}^\infty$ nor $(x_n)_{n=1}^\infty$ can be convergent to $x$ in the $\tau_\xi$ topology. In the second case, there exists a subsequence $(y_n)_{n=1}^\infty$ of $(x_n)_{n=1}^\infty$ such that $(y_n-x)_{n=1}^\infty$ is a weakly null $\ell_1^\xi$-spreading model, as are all of its subsqeuences. Then the sequence $(y_n)_{n=1}^\infty$ has no $\xi$-convergent subsequence, since $(y_n-x)_{n=1}^\infty$ has no $\xi$-weakly null subsequence. In this case $T=\{y_n: n\in\nn\}$ is $\tau_\xi$-closed. Then $U=X\setminus T$ is $\tau_\xi$-open and contains the sequence $(y_n)_{n=1}^\infty$. From this it follows that the original sequence $(x_n)_{n=1}^\infty$ is not $\tau_\xi$-converent to $x$.

\end{proof}

The following result includes a standard fact from topology, together with a minor variation thereof. Since the variation is not standard, we include the proof.

\begin{lemma} Let $K$ be a topological space which is either compact or sequentially compact. Let $Y$ be a Hausdorff topological space such that a subset of $Y$ is compact if and only if it is sequentially compact. Suppose $f:K\to Y$ is a continuous injection. Then $f$ is a homeomorphism of $K$ with $f(K)$.  In either case, $K$ is both sequentially compact and compact. 

\label{magic}
\end{lemma}

\begin{proof} Note that if $K$ is compact, so is $f(K)$, and if $K$ is sequentially compact, so is $f(K)$. Since compactness and sequential compactness are equivalent in $Y$, in either case, $f(K)$ is both compact and sequentially compact.  Since $Y$ is Hausdorff, $f(K)$ is closed. From this it follows that a subset of $f(K)$ is compact if and only if it is sequentially compact. Therefore by relabeling, we may assume $Y=f(K)$ and assume $f$ is a bijection.  In the case $K$ is compact and the case that $K$ is sequentially compact, it is sufficient to show that if $C\subset K$ is closed, $f(C)$ is closed. 

We first provide the proof in the case that $K$ is compact. This is standard, but we provide the proof to help illustrate the case in which $X$ is sequentially compact.   Fix $C\subset K$ closed. Since $C$ is a closed subset of a compact set, $C$ is compact, and so is $f(C)$.  Since $Y$ is Hausdorff, $f(C)$ is closed. 

Now assume $K$ is sequentially compact and fix $C\subset K$ closed. Then $C$ is also sequentially compact, as is $f(C)$. In this case, $f(C)$ is also compact, and therefore closed. This completes the proof that $f$ is a homeomorphism. 

Since $K$ is homeomorphic to $f(K)$, which is both compact and sequentially compact, $K$ is both compact and sequentially compact. 

\end{proof}

\begin{corollary} Fix $0<\xi<\omega_1$. \begin{enumerate}[(i)]\item A subset $K$ of $X$ is $\xi$-weakly compact if and only if it is sequentially compact in the $\tau_\xi$ topology if and only if it is compact in the $\tau_\xi$ topology.  In this case, the $\tau_\xi$ and weak topologies on $K$ coincide. \item A subset $K$ of $X$ is $\xi$-weakly precompact if and only if it is relatively compact in the $\tau_\xi$ topology, and in this case $\overline{K}^\text{\emph{weak}}=\overline{K}^{\tau_\xi}$, and these two sets are homeomorphic with their weak and $\tau_\xi$ topologies.   \end{enumerate}

\end{corollary}

\begin{proof}$(i)$ Recall that a sequence $(x_n)_{n=1}^\infty\subset X$ is $\xi$-weakly convergent to $x\in X$ if and only if is it convergent to $x$ in the $\tau_\xi$ topology.  With this fact, it is immediate from the definition of $\xi$-weakly compact that $K\subset X$ is $\xi$-weakly compact if and only if it is sequentially compact in the $\tau_\xi$ topology.    

Now assume that $K\subset X$ is either compact or sequentially compact in the $\tau_\xi$ topology. Let $Y=K$ endowed with its weak topology. Since the $\tau_\xi$ topology is finer than the weak topology, $f:K\to Y$ is a continuous bijection. By Lemma \ref{magic}, $K$ is both compact and sequentially compact, and the $\tau_\xi$ and weak topologies coincide on $K$. This concludes $(i)$.  We note that in order to apply Lemma \ref{magic}, we use the Eberlein-\v{S}mulian Theorem to deduce that a subset of $Y$ is compact if and only if it is sequentially compact.

$(ii)$ For any $K\subset X$, since the $\tau_\xi$ topology is finer than the weak topology, $\overline{K}^{\tau_\xi}\subset \overline{K}^{\text{weak}}$.    

If $K$ is $\xi$-weakly precompact, then $\overline{K}^\text{weak}$ is $\xi$-weakly compact by Proposition \ref{barbarian}, and the weak and $\tau_\xi$ topologies are the same on $K$ by $(i)$. This means the $\tau_\xi$ and weak topologies are the same on $K$, $ \overline{K}^{\tau_\xi}$, and $\overline{K}^{\text{weak}}$. In particular, these last two sets are equal. 

Now if $K$ is relatively compact in the $\tau_\xi$ topology, then $\overline{K}^{\tau_\xi}$ is compact in the $\tau_\xi$ topology. It is therefore weakly compact, and weakly closed, so $\overline{K}^\text{weak}\subset \overline{K}^{\tau_\xi}$.  This shows that $\overline{K}^{\tau_\xi}=\overline{K}^{\text{weak}}$. Since the $\tau_\xi$ and weak topologies coincide on $\overline{K}^{\tau_\xi}$, they coincide on $\overline{K}^{\text{weak}}$.

\end{proof}

\section{Convexity}

In this section, we discuss the interplay between  the $\tau_\xi$ topology and convexity. For finite, non-zero values of $\xi$, the $\tau_\xi$ topology need not be locally convex.  We modify an example from \cite{LA} to give an example of a Banach space $Y_k$ with normalized, $1$-unconditional, $k+1$-weakly null basis such that the closed, convex hull of the basis is not $k+1$-weakly compact. 

In what follows, for $y\in c_{00}$ and $I\subset \nn$, $Iy$ is the projection of $y$ onto $\text{span}\{e_n: n\in I\}$.   For the proof, we recall the definition for $m,n\in\nn$ $$\mathcal{S}_m[\mathcal{S}_n]=\{\varnothing\}\cup \Bigl\{\bigcup_{i=1}^t E_i: \varnothing\neq E_i, E_1<\ldots <E_t, E_i\in \mathcal{S}_n, (\min E_i)_{i=1}^t\in \mathcal{S}_m\Bigr\}.$$  
It is a standard fact concerning $\mathcal{S}_n$, $n<\omega$, that $\mathcal{S}_m[\mathcal{S}_n]=\mathcal{S}_{m+n}=\mathcal{S}_n[\mathcal{S}_m]$ for all non-negative integers $m,n$.

\begin{lemma} Let $Y$ be a Banach space such that the canonical $c_{00}$ basis is a normalized, unconditional Schauder basis for $Y$. For $k\in\nn$, let us define the norm $$\|y\|_k= \sup \Bigl\{\sum_{i=1}^t \|I_i y\|_Y: I_1<\ldots <I_t, (\min I_i)_{i=1}^t\in\mathcal{S}_k\Bigr\}$$ and let $Y_k$ be the completion of $c_{00}$ with respect to this norm.  \begin{enumerate}[(i)]\item If $(y_n)_{n=1}^\infty$ is a convex block sequence of the $c_{00}$ basis which is an $\ell_1^1$-spreading model in $Y$, then $(y_n)_{n=1}^\infty$ is an $\ell_1^{k+1}$-spreading model in $Y_k$.  \item If the canonical $c_{00}$ basis is $1$-weakly null in $Y$, then it is $k+1$-weakly null in $Y_k$. \end{enumerate}
\label{fin}

\end{lemma}

\begin{proof} For both parts, note that the canonical $c_{00}$ basis is also normalized and unconditional in $Y_k$. 

$(i)$ Assume $(y_n)_{n=1}^\infty$ is a convex block sequence of the canonical $c_{00}$ basis which is an $\ell_1^1$-spreading model in $Y$. Let $$\ee=\inf \{\|y\|: E\in \mathcal{S}_1, y=\sum_{n\in E} a_ny_n, \sum_{n\in E}|a_n|=1\}>0.$$  Since the canonical $c_{00}$ basis is normalized in $Y_k$, $\|y_n\|_k\leqslant 1$ for all $n\in\nn$, and $(y_n)_{n=1}^\infty$ is bounded.    Fix $E\in \mathcal{S}_{k+1}$ and scalars $(a_i)_{i\in E}$.    Since   $\mathcal{S}_{k+1}=\mathcal{S}_k[\mathcal{S}_1]$, we may write $$E=\bigcup_{i=1}^t E_i$$ for some non-empty $E_1<\ldots <E_t$, $E_i\in \mathcal{S}_1$, $(\min E_i)_{i=1}^t\in \mathcal{S}_k$.  For each $1\leqslant i\leqslant t$, let $$I_i= [\min \supp(y_{\min E_i}), \max \supp(y_{\max E_i})].$$  Then $\min I_i\geqslant \min E_i$, so $(\min I_i)_{i=1}^t\in \mathcal{S}_k$.  Furthermore, $I_1<\ldots <I_t$.    Therefore \begin{align*} \|\sum_{n\in E} a_ny_n\|_k &\geqslant \sum_{i=1}^t \|I_i\sum_{n\in E}a_ny_n\|_Y = \sum_{i=1}^t \|\sum_{n\in E_i} a_ny_n\|_Y \geqslant \ee\sum_{i=1}^t \sum_{n\in E_i}|a_n| \\ & = \ee \sum_{n\in E}|a_n|.  \end{align*}   This shows that $(y_n)_{n=1}^\infty$ is an $\ell_1^{k+1}$-spreading model in $Y_k$. 

$(ii)$  Since the canonical $c_{00}$ basis $(e_n)_{n=1}^\infty$ is bounded and unconditional in $Y_k$, it is either weakly null or it has a subsequence equivalent to the canonical $\ell_1$ basis.  Thus if we can show that $(e_n)_{n=1}^\infty\subset Y_k$ has no subsequence which is an $\ell_1^{k+1}$-spreading model, we will know that it has no subsequence equivalent to the canonical $\ell_1$ basis, and must be weakly null. In this case, we will have shown that $(e_n)_{n=1}^\infty\subset Y_k$ is weakly null with no subsequence that is an $\ell_1^{k+1}$ spreading model, and is therefore $k+1$-weakly null.    

Note that the condition that the canonical basis of $Y_k$ has no subsequence which is an $\ell_1^{k+1}$-spreading model implies that for any $m_1<m_2<\ldots$, $\ee>0$, and $j\in\nn$, we can find $E\in \mathcal{S}_{k+1}$ such that $\min E>j$ and scalars $(a_i)_{i\in E}$ such that $\sum_{i\in E}|a_i|=1$ and $\|\sum_{i\in E}a_ie_{m_i}\|_k<\ee$. To see this, apply the preceding paragraph to the sequence $m_{j+1}, m_{j+2}, \ldots$ to find $F\in \mathcal{S}_\xi$ and scalars $(b_i)_{i\in F}$ such that $\sum_{i\in F}|b_i|=1$ and $\|\sum_{i\in F}b_i e_{m_{i+j}}\|_k<\ee$. Then let $E=\{i+j: i\in F\}\in \mathcal{S}_{k+1}$ and $a_{i+j}=b_i$ for $i\in F$. 

Let $Y_0=Y$. By the preceding paragraph, for $k\in \nn\cup \{0\}$, the condition that $(e_n)_{n=1}^\infty$ is $k+1$-weakly null in $Y_k$ is equivalent to the condition that it has no subsequence which is an  $\ell_1^{k+1}$-spreading model, which is equivalent to the condition that for any $m_1<m_2<\ldots$ and $\ee>0$, there exist  $E\in \mathcal{S}_\xi$ and $(a_i)_{i\in E}$ such that $\sum_{i\in E}|a_i|=1$ and $\|\sum_{i\in E}a_i e_{m_i}\|_k<\ee$. We prove this by induction on $k\in \nn\cup \{0\}$, where the base case $k=0$ is true by hypothesis.   Now assume that for $k\in\nn$, $(e_n)_{n=1}^\infty$ is $k$-weakly null in $Y_{k-1}$.  Fix $m_1<m_2<\ldots$ and $\ee>0$. We may assume $\ee<1$.   By recursive applications of the preceding paragraph, and with $p_0=1$ and $p_n=\max \{m_i: i\in E_n\}$ for $n\in\nn$, we may find $E_1<E_2<\ldots$, $E_n\in \mathcal{S}_k$, and scalars $(a_i)_{i\in E_n}$ such that $\sum_{i\in E_n}|a_i|=1$ and $$\|\sum_{i\in E_n} a_i e_{m_i}\|_{k-1} < \frac{\ee/2}{ p_{n-1}}.$$    Note that $1=p_1<p_2<\ldots$.    Now fix $m\in \nn$ such that $1/m<\ee/2$.  We will show that $$\|\sum_{i=m+1}^{2m} \frac{1}{m}\sum_{j\in E_i} a_j e_{m_j}\|_k<\ee.$$  Since $\cup_{i=m+1}^{2m} E_i\in \mathcal{S}_{k+1}$ and $\sum_{i=m+1}^{2m} \frac{1}{m}\sum_{j\in E_i} |a_j|=1$, this will finish the proof.  Fix intervals $I_1<\ldots <I_t$ such that $J=(\min I_i)_{i=1}^t\in \mathcal{S}_{k+1}$ and $$\|\sum_{i=m+1}^{2m} \frac{1}{m}\sum_{j\in E_i} a_j e_{m_j}\|_k= \sum_{r=1}^t \Bigl\|I_r \Bigl(\sum_{i=m+1}^{2m} \frac{1}{m}\sum_{j\in E_i} a_j e_{m_j}\Bigr)\Bigr\|_Y.$$  Since $\mathcal{S}_{k+1}=\mathcal{S}_1[\mathcal{S}_k]$, we may write $J=\cup_{i=1}^s J_i$ with $J_1<\ldots <J_s$, $\varnothing\neq J_i\in \mathcal{S}_k$, and $(\min J_i)_{i=1}^s\in \mathcal{S}_1$. By omitting any superfluous $J_i$ set and relabeling, we may assume that $$J_1\sum_{i=m+1}^{2m} \frac{1}{m}\sum_{j\in E_i} a_j e_{m_j}\neq 0.$$ The  membership $(\min J_i)_{i=1}^s\in \mathcal{S}_1$ means that $s\leqslant J_1$.    Now let $l$ be the minimum $i$ such that $m+1\leqslant i\leqslant 2m$ and $J_1 \sum_{j\in E_i} a_j e_{m_j}\neq 0$.  Then $$s\leqslant \min J_1 \leqslant p_l.$$   Note also that for each $l< i\leqslant 2m$, $$\sum_{r=1}^s \|J_r\sum_{j\in E_i} a_j e_{m_j}\|_Y \leqslant s\|\sum_{j\in E_i}a_je_{m_j}\|_{k-1} \leqslant s \cdot \frac{\ee/2}{p_{i-1}} \leqslant \ee/2.$$  Then \begin{align*} \|\sum_{i=m+1}^{2m} \frac{1}{m}\sum_{j\in E_i} a_j e_{m_j}\|_k & = \sum_{r=1}^s \Bigl\|I_r \Bigl(\sum_{i=l}^{2m} \frac{1}{m}\sum_{j\in E_i} a_j e_{m_j}\Bigr)\Bigr\|_Y \\ & \leqslant \frac{1}{m}\|\sum_{j\in E_l}a_je_{m_j}\|_{\ell_1} + \frac{1}{m} \sum_{i=l+1}^{2m} \ee/2<\ee/2+\ee/2=\ee.\end{align*}

\end{proof}

We next recall a construction from \cite{LA} which gives the requisite $Y$ for the preceding lemma. 

\begin{theorem} There exists a hereditary set $\mathcal{F}$ of finite subsets of $\nn$ which contains all singletons such that, if $Y$ is the completion of $c_{00}$ with respect to the norm $$\|\sum_{i=1}^\infty a_ie_i\|_Y= \sup \{\sum_{i\in E}|a_i|: E\in \mathcal{F}\},$$ then  the canonical $c_{00}$ basis is $1$-weakly null in $Y$, but the basis of $Y$ admits a convex block sequence which is an $\ell_1^1$-spreading model. 

\label{LA}
\end{theorem}

\begin{theorem} For every $k\in\nn$, there exists a hereditary set $\mathcal{F}_k$ containing all singletons such that the Banach space $Y_k$ which is the completion of $c_{00}$ with respect to the norm $$\|\sum_{i=1}^\infty a_ie_i\|_Y= \sup \{\sum_{i\in E}|a_i|: E\in \mathcal{F}_k\}$$ has the property that its canonical basis is $k+1$-weakly null, but admits a convex block sequence $(y_n)_{n=1}^\infty$ of the basis which is an $\ell_1^{k+1}$-spreading model. Furthermore, we may choose the sequence $(y_n)_{n=1}^\infty$ to be independent of $k$. 

\label{build}

\end{theorem}

\begin{proof} For the $k=0$ case, we let $\mathcal{F}_0=\mathcal{F}$ and $Y_0=Y$, the space from Theorem \ref{LA}.  For $k>0$, we let $$\mathcal{F}_k=\Bigl\{\bigcup_{i=1}^t E_i: (\min E_i)_{i=1}^t\in \mathcal{S}_k, E_1<\ldots <E_t, E_i\in \mathcal{F}\Bigr\}.$$  It is easily verified that $Y_k$ as defined in the corollary is isometrically the same as $Y_k$ as defined in Lemma \ref{fin}, and therefore $(e_n)_{n=1}^\infty$ and $(y_n)_{n=1}^\infty$ have the requisite properties.

\end{proof}

This raises the following question. 

\begin{question} For $\xi<\omega_1$, does there exist $\zeta=\zeta(\xi)<\omega_1$ such that for any Banach space $X$ and any $\xi$-weakly compact subset $K$ of $X$, the closed, convex hull of $K$ is $\zeta$-weakly compact?

\end{question}

\begin{corollary} For each $k\in\nn$, $C(2^\nn)$ contains a $k$-weakly compact subset whose closed, convex hull is not $k$-weakly compact. Furthermore, the topology $\tau_k$ on $C(2^\nn)$ is not locally convex.

\label{pelc}
\end{corollary}

\begin{proof} Fix $k\in\nn$ and let $(f_n)_{n=1}^\infty$ be a sequence in $C(2^\nn)$ equivalent to the basis of $Y_{k-1}$. Let $C=\{f_n:n\in\nn\}$. Then $C$ is $k$-weakly compact, while its closed, convex hull is not. This is because $\text{co}(C)$ contains a weakly null $\ell_1^k$-spreading model, so $\text{co}(C)$ is not $k$-weakly precompact by Proposition \ref{crit}. The second statement of the corollary follows from the fact that $(f_n)_{n=1}^\infty\subset C(2^\nn)$ is convergent to $0$ in the $\tau_k$ topology but it admits a convex block sequence which is not convergent to $0$ in the $\tau_k$ topology.

\end{proof}

\section{Examples of $\xi$-Schur Banach spaces}

The canonical bases of $c_0$ and $\ell_p$ for $1<p<\infty$ are normalized and $1$-weakly null. Therefore these spaces fail to be $1$-Schur. Furthermore, any space which contains an isomorphic copy of one of these spaces also fails to be $1$-Schur.   On the other hand, $\ell_1$ is a Schur space. Therefore all classical Banach spaces lie on one extreme of the Schur hierarchy or the other.

It is well-known that any Schur Banach space must be $\ell_1$ saturated. In \cite{azimi}, the authors gave an example of an $\ell_1$ saturated Banach space which does not have the Schur property.  Fix $1=\alpha_1>\alpha_2>\ldots$ such that $\lim_n \alpha_n=0$ and $\sum_{n=1}^\infty \alpha_n=\infty$.  The authors of \cite{azimi} showed that the space $U$ which is the completion of $c_{00}$ under the norm $$\|x\|=\sup \Biggl\{\Bigl|\sum_{i=1}^n \alpha_i \sum_{j\in I_i} x_j\Bigr|: n\in\nn, I_1<\ldots <I_n, I_i\text{\ an interval}\Biggr\}$$   is $\ell_1$ saturated, the canonical basis $(e_i)_{i=1}^\infty$ of $U$ is weakly Cauchy, and $$\|\sum_{i\in E} e_{2i}-e_{2i-1}\|= \sum_{i=1}^{2|E|} \alpha_i$$ for any finite $E\subset \nn$.  Since $\lim_n \frac{1}{n}\sum_{i=1}^{2n}\alpha_i=0$, it follows that $(e_{2i}-e_{2i-1})_{i=1}^\infty$ is weakly null with no $\ell_1^1$-spreading model subsequence. Therefore the Azimi-Hagler space $U$ is $\ell_1$ saturated while not being $1$-Schur.

The Tsirelson space $T_{\xi, \vartheta}$, defined for $0<\xi<\omega_1$ and $0<\vartheta<1$, is the completion of $c_{00}$ with respect to the implicitly defined norm $$\|x\|= \max\Biggl\{\|x\|_{c_0}, \sup\Bigl\{\vartheta\sum_{i=1}^n \|I_i x\|: I_1<\ldots <I_n, (\min I_i)_{i=1}^n\in \mathcal{S}_\xi\Bigr\}\Biggr\}.$$   It is easy to see that $T_{\xi, \vartheta}$ is $\xi$-Schur. Indeed, if $(x_n)_{n=1}^\infty$ is a block sequence with $a=\inf\|x_n\|$ and if $I_1<I_2<\ldots$ are such that $\min I_i=\min \text{supp}(x_i)$ and $\max I_i=\max \text{supp}(x_i)$, $$\|\sum_{i\in E} a_ix_i\|\geqslant \vartheta \sum_{j\in E} \|I_j \sum_{i\in E}a_ix_i\| \geqslant \vartheta a\sum_{i\in E}|a_i|$$ for any $E\in \mathcal{S}_\xi$. From iterating this observation and using deeper combinatorial properties of the Schreier spaces, $T_{\xi, \vartheta}$ is $\xi n$-Schur for all natural number $n\in\nn$.  It is a consequence of  \cite[Theorem 5.19]{JuddOdell} that $T_{\xi, \vartheta}$ cannot have an $\ell_1^{\xi \omega}$-spreading model. Since $T_{\xi, \vartheta}$ is reflexive, $T_{\xi, \vartheta}$ fails to be $\xi\omega$-Schur. By standard properties of ordinals, if $\gamma$ is the maximum ordinal such that $\omega^\gamma\leqslant \xi$, $\xi\omega=\omega^{\gamma+1}$. Thus for every $\gamma<\omega_1$, we have found an example of a Banach space, namely $T_{\omega^\gamma, 1/2}$, which is $\zeta$-Schur for all $\zeta<\omega^{\gamma+1}$, but which fails to be $\omega^{\gamma+1}$-Schur.

For $0<\xi<\omega_1$, fix a sequence $\xi_n\uparrow \omega^\xi$ and a sequence $\vartheta_1>\vartheta_2>\ldots$ such that $\sum_{n=1}^\infty \vartheta_n <1$.  Let $Z_\xi$ be the completion of $c_{00}$ with respect to the norm $$\|x\|=\max\Bigl\{\|x\|_{c_0},\Bigl(\sum_{n=1}^\infty \|x\|_n^2\Bigr)^{1/2}\Bigr\},$$ where $$\|x\|_n= \sup\Bigl\{\vartheta_n\sum_{i=1}^t \|I_ix\|: I_1<\ldots <I_t, (\min I_i)_{i=1}^t \in \mathcal{S}_{\xi_n}\Bigr\}.$$    It was shown in \cite{CN} that $Z_\xi$ is $\zeta$-Schur for each $\zeta<\omega^\xi$, but $Z_\xi$ is not $\omega^\xi$-Schur.   Thus for every $\xi<\omega_1$, we have found a Banach space which is $\zeta$-Schur for every $\zeta<\omega^\xi$ and which fails to be $\omega^\xi$-Schur.   As remarked previously, for $\xi<\omega_1$ and $\omega^\xi<\zeta<\omega^{\xi+1}$, $\textsf{V}_\zeta=\textsf{V}_{\omega^\xi}$, these examples represent the sharpest possible control over such examples of spaces which are $\xi$-Schur for specified $\xi$.

\end{document}